\documentclass[12pt]{amsart}
\usepackage{amssymb}
\usepackage[all]{xy}
\usepackage{fancyhdr}
\usepackage{nomencl}
\usepackage{anysize}
\marginsize{0.9in}{0.9in}{0.7in}{0.7in}

\usepackage{bbm}
\usepackage{latexsym}
\usepackage{setspace}
\usepackage{CJK}
\usepackage{graphicx}
\usepackage{url}
\usepackage{float}
\usepackage{nomencl}
\usepackage{amsmath,amsthm,amsfonts,amssymb}
\usepackage{enumerate}
\usepackage{hyperref}
\usepackage{mathrsfs}

\DeclareMathOperator{\var}{Var }

\DeclareMathOperator{\supp}{supp}

\makeatletter
\newcommand*{\Relbarfill@}{\arrowfill@\Relbar\Relbar\Relbar}
\newcommand*{\xeq}[2][]{\ext@arrow 0055\Relbarfill@{#1}{#2}}
\makeatother

\theoremstyle{plain}
\newtheorem{thm}{Theorem}[section] 

\newtheorem{lem}[thm]{Lemma}

\newtheorem{prop}[thm]{Proposition}
\newtheorem{cor}[thm]{Corollary}

\theoremstyle{definition}
\newtheorem{defn}[thm]{Definition}

\newtheorem{rem}[thm]{Remark}

\DeclareMathOperator{\sub}{Sub}

\newcommand*\diff{\mathop{}\!\mathrm{d}}

\begin{document}
	
	\title{On the Strong Restricted Isometry Property of Bernoulli Random Matrices}
	\author{Ran Lu}
	\thanks{Research supported in part by NSERC Canada under Grant
		RGP 228051. 
	}
	\address{Department of Mathematical and Statistical Sciences, University of Alberta, Edmonton, AB, Canada, T6G 2G1}
	\email{rlu3@ualberta.ca}

	\begin{abstract}
		The study of the restricted isometry property (RIP) of corrupted random matrices is particularly important in the field of compressed sensing (CS) with corruptions. If a matrix still satisfies the RIP after that a certain portion of rows are erased, then we say that this matrix has the strong restricted isometry property (SRIP). In the field of compressed sensing, random matrices which satisfy certain moment conditions are of particular interest. Among these matrices, those with entries generated from i.i.d. Gaussian or symmetric Bernoulli random variables are often typically considered. Recent studies have shown that matrices with entries generated from i.i.d. Gaussian random variables satisfy the SRIP under arbitrary erasure of rows with high probability. In this paper, we study the erasure robustness property of Bernoulli random matrices. Our main result shows that with overwhelming probability, the SRIP holds for Bernoulli random matrices. Moreover, our analysis leads to a robust version of the famous Johnson-Lindenstrauss lemma for Bernoulli random matrices.
	\end{abstract}

	\keywords{Compressed sensing, Bernoulli random matrices, Strong restricted isometry property, Johnson-Lindenstrauss lemma.}
	\subjclass[2010]{62G35, 42C15} 
	
	\maketitle

	\begin{spacing}{1}
		\begin{flushleft}
			
			\section{Introduction}
			
			\subsection{Background on Compressed Sensing}

			Consider the following problem: In the sensor network, we transmit a signal through independent channels and the center hub receives observations from each channel for further analysis. In practice, it is typical that some channels fail to send the correct measurements, and thus we can only obtain the corrupted data. To deal with problems like this, methods of reconstructing a signal from sampling observations are of great interest to researchers in the fields of engineering. In general, it is impossible to recover a signal if there is nothing known about the signal or the measurement. However, with prior knowledge of the signal, it is possible to perform the recovery with negligible or even zero error. Over the past years, researchers have been refreshing their understandings on the relevancy and practicability on the prior assumptions for the input signal and the measurement, and have developed several powerful tools to study the signal recovery problem.\\
			
			\vspace{0.2cm}
			
			In signal processing, \emph{compressed sensing} (CS) is a technique for signal recovery via solving certain linear systems. In compressed sensing \emph{without corruption}, the general acquisition setting is represented as $y=Ax$, where $x\in\mathbb{R}^m$ is a \emph{signal}, $A\in\mathbb{R}^{n\times m}$ is a \emph{sensing matrix} ( here $\mathbb{R}^{n\times m}$ denotes the set of all $n\times m$ real matrices), and $y\in\mathbb{R}^n$ is the \emph{measurement}. In general, it is extremely difficult to recover the signal $x$ if we know nothing about $x$ itself and the sensing matrix $A$. Around 2004, it was shown by Cand\'es, Romberg and Tao in \cite{sr}, and Donoho in \cite{csd}, that if a signal satisfies certain sparsity condition plus additional assumptions on the sensing matrix, then the reconstruction can be accomplished. The ideas established in those papers provided the foundations of CS.\\
			
			\vspace{0.2cm}
			
			A key concept introduced in CS is the \emph{restricted isometry property} (RIP), which gives a characterization of ``the almost norm preserving property" of a matrix. This property was first introduced by Cand\'es and Tao in \cite{dl}.
			
			\begin{defn}A matrix $A\in\mathbb{R}^{n\times m}$ satisfies the \emph{restricted isometry property} of order $s$ with $s\in \{1,\dots,m\}$, if there exists $\delta\in[0,1)$ such that
				$$(1-\delta)\|x\|_2^2\leq\|Ax\|_2^2\leq(1+\delta)\|x\|_2^2$$
				for all $x\in \mathbb{R}^m$ with $\|x\|_0\leq s$, where $\|x\|_0$ denotes the number of non-zero entries of $x$. Moreover we define the $s$-RIP constant of $A$ via
				$$\delta_{s,A}:=\inf\{\delta\in[0,1):(1-\delta)\|x\|_2^2\leq\|Ax\|_2^2\leq(1+\delta)\|x\|_2^2\text{ for all }x\in \mathbb{R}^m\text{ with }\|x\|_0\leq s \}.$$
			\end{defn}
			
			It was shown by Cand\'es and Tao in \cite{dl} that one can exactly recover a sparse signal $x$ with $\|x\|_0\leq s$, if the $s$-RIP constant of the sensing matrix $A$ satisfies certain condition. In practice, matrices which satisfy the RIP are usually generated from random variables. There are many types of random matrices which have the RIP with overwhelming probability. We list few examples below:
			 \begin{itemize}
			 	\item Matrices with entries drawn from i.i.d. Gaussian random variables.
			 	
			 	\item Matrices with entries drawn from i.i.d. symmetric Bernoulli random variables.
			 	
			 	\item Matrices with rows drawn randomly from a discrete Fourier transform matrix.
			 	
			 	\item Matrices with entries drawn from i.i.d. sub-gaussian or sub-exponential random variables.
			 \end{itemize} 
		 For more examples and details, see e.g. \cite{da,sprip,sr,ht,csrm,os}.\\
		 
		 \vspace{0.2cm}

			  One the other hand, the RIP is closely related to the well-known \emph{Johnson-Lindenstrauss(JL) lemma}, which was first stated and proved by Johnson and Lindenstrauss in \cite{jl} under an abstract setting. Here we provide a specific version of the JL lemma, which is the one we often use in CS:
			 \begin{lem}[\textbf{Johnson-Lindenstrauss lemma}](\cite[Lemma 4.1]{sprip})
			 	Let $\epsilon\in(0,1)$ and $N\in\mathbb{N}$. For every finite subset $Q$ of $\mathbb{R}^N$, let $k_Q\in\mathbb{N}$ be such that $k_Q>O(\epsilon^{-2}\ln(|Q|))$, where $|Q|$ denotes the cardinality of the set $Q$. Then there exists a Lipschitz function $f:\mathbb{R}^N\to\mathbb{R}^{k_Q}$ such that
			 	$$(1-\epsilon)\|u-v\|_{\ell^2(\mathbb{R}^N)}^2\leq\|f(u)-f(v)\|_{\ell^2(\mathbb{R}^{k_Q})}^2\leq (1+\epsilon)\|u-v\|_{\ell^2(\mathbb{R}^N)}^2,\quad\forall u,v\in Q.$$
			 	\end{lem}
			  The JL lemma says that a discrete set of points in a high dimensional Euclidean space can be embedded into a low dimensional space, in a way such that the distances between points are nearly preserved. From \cite[Theorem 5.2]{sprip}, one can see that the JL lemma implies the RIP.\\

			\subsection{Compressed Sensing with Corruptions}
			
			A natural generalization of CS is \emph{CS with corruptions}. In this case, some elements of the measurement are corrupted. We formulate the model in this setting as follows: Given a signal $x\in\mathbb{R}^m$ with certain sparsity condition and a sensing matrix $A\in\mathbb{R}^{n\times m}$. We only receive $A_Tx$ as the observation, where $T\subseteq\{1,\dots, n\}$ is unknown and $A_T$ denotes the sub-matrix by keeping the rows of $A$ with indices in $T$. In order to reconstruct the signal accurately and efficiently, it is important to verify whether $A_T$ satisfies the RIP. This leads to the concept of the \emph{strong restricted isometry property} (SRIP), which plays a central role in the study of the robustness of a matrix. 
			
			\begin{defn}\label{sripdef}A matrix $A\in\mathbb{R}^{n\times m}$ is said to satisfy the \emph{strong restricted isometry property} (SRIP) of order $s$ and level $[\theta, \omega,\beta]$ with $0<\theta\leq 1\leq\omega<2$ and $\beta\in[0,1)$ if
				$$\theta \|x\|_2^2\leq\|A_Tx\|_2^2\leq\omega\|x\|_2^2$$
				holds for all $x\in\mathbb{R}^m$ with $\|x\|_0\leq s$ and all $T\subseteq\{1,\dots,n\}\text{ with }|T^c|\leq\beta n.$
			\end{defn}

			The concept of SRIP was introduced in \cite{SRIP1}, and the authors used the terminology "democratic" to describe that all measurement elements are equally important. To be more specific, an $n\times m$ matrix $B$ is said to be $(\tilde{n},s,\delta)$-democratic with $\tilde{n}\in\{1,\dots,n\}$, $s\in\{1,\dots, m\}$ and $0<\delta<1$ if
			$$(1-\delta)\|x\|_2^2\leq \|A_Tx\|_2^2\leq(1+\delta)\|x\|_2^2,\quad \forall T\subseteq\{1,\dots,n\}\text{ with }|T|\geq\tilde{n},\quad\forall x\in\mathbb{R}^m\text{ with }\|x\|_0\leq s.$$ 
			The erasure robostness property for Gaussian random matrices was investigated in \cite{SRIP1}, and the authors made the comment that the same analysis can be performed on sub-gaussian random matrices. It's proved in \cite{SRIP1} that, an $n\times m$ Gaussian (sub-gaussian) random matrix is $(\tilde{n},s,\delta)$-democratic with high probability provided that
			\begin{equation}\label{const}n=C(s+n-\tilde{n})\ln\left(\frac{m+n}{s+n-\tilde{n}}\right),\end{equation}
			where $C>0$ is an absolute constant. One can see that the constraint \eqref{const} is slightly unnatural, comparing with the constraint $n=O(s\ln(m/s))$ as in traditional CS. Moreover, from the analysis performed in \cite{SRIP1}, it's unclear whether the proportion of rows which are allowed to be corrupted depends on the size of the sensing matrix or the dimension of the measurement vector. Later, the definition of the SRIP, which is Definition\autoref{sripdef}, was introduced in \cite{srip} This definition is more preferred because it sticks to the ratio rather than the number of rows being corrupted, which makes more sense for models with huge amount of data. It's proved in \cite{srip} that Gaussian random matrices have the SRIP of certain level with overwhelming probability. In \cite{rp}, further results about the robustness property of Gaussian random matrices were proved, including the robust version of the JL lemma for Gaussian random matrices.\\
			
			\vspace{0.2cm}
			In this paper, we concentrate our effort on Bernoulli random matrices, as they are used quite often as sensing matrices in CS (\cite{csber1, csber2, csber3, csber4}). A Bernoulli random matrix is defined as the following:
			
			\begin{defn}We say that a real random variable $X$ has the symmetric Bernoulli distribution, if $X$ only takes values $\pm 1$ with
				$$\mathbb{P}(X=1)=\mathbb{P}(X=-1)=\frac{1}{2}.$$
			\end{defn}
			
			A Bernoulli random matrix is a random matrix with entries drawn from i.i.d. symmetric Bernoulli random variables. It was proved that a Bernoulli random matrix is singular with positive probability (see \cite{tv}). The study on the singularity of a Bernoulli random matrix is closely related to its erasure robustness property. From the fact that a Bernoulli matrix is singular with positive probability, one may expect an upper bound on the portion of rows which can be erased, such that the corrupted Bernoulli random matrix will still satisfy the RIP with high probability. In fact, it is not difficult to show that the portion of rows erased cannot reach $\frac{1}{2}$. However, it is still unkonwn whether or not $\frac{1}{2}$ is the optimal upper bound for the erasure ratio.

			\subsection{Contributions and the Structure of the Paper}
		
		In this paper, we will prove that a Bernoulli random matrix satisfies the SRIP of certain order and level.	Our study relies on the fact that a Bernoulli random variable belongs to the sub-gaussian class, which will be discussed in the next section. This indeed makes our study to share certain ingredients with \cite{SRIP1}, but still several new elements are involved in this paper. The key ingredient to yield the main result in \cite{SRIP1} is the classical concentration of measure property for sub-gaussian random matrices, which also plays an important role in our study. However, with the classical concentration of measure property along, one may not expect to do better than \cite{SRIP1}, particularly one cannot deduce the result on the erasure ratio. Moreover, we have illustrated in the previous subsection that an unnatural constraint \eqref{const} has to be imposed in order to yield the main result in \cite{SRIP1}. Thus in order to do better, we need to use several new tools, including a Lipschitz type concentration inequality, some order statistical techniques, and classical results on Bernoulli random variables/matrices. It's also interesting to note that, our main theorem tells that an $n\times m$ Bernoulli random matrix $\Phi$ satisfies the SRIP with high probability of order $s$ and of certain level, where we only impose the constraint $n=O(s\ln(m/s))$, just as in traditional CS.
			
				\vspace{0.2cm}
			
			The structure of this paper is organized as follows: In section 2, a brief review on symmetric Bernoulli random variables will be given. We will provide some concentration inequalities and order statistics results, which are essential and will be used later. In section 3, the erasure robustness property of Bernoulli random matrices will be concretely treated. We will focus on the following question: if no more than $\beta n$  rows can be erased from an $n\times m$ Bernoulli random matrix $\Phi$ with some $\beta\in(0,1)$, how large can $\beta$ be so that the corrupted matrix still has the RIP with high probability? We will provide a lower bound on the erasure ratio $\beta$. As a direct consequence, we will establish the SRIP and a robust version of the JL lemma for Bernoulli random matrices. Finally, a summary and further discussion will be given in Section 4.\\

			\section{On Symmetric Bernoulli Random Variables}
			
			In this section, we provide some necessary background on Bernoulli random variables.
			
			\subsection{Sub-gaussian Random Variables}
			
			One of the most important properties of a Bernoulli random variable is that, it belongs to the sub-gaussian class (the definition will be provided later). Many of the auxiliary results on Bernoulli random variables are actually developed from the sub-gaussian property. Therefore, it's better for us to review some basics on sub-gaussian random variables. All definitions and results mentioned in this subsection are known. For further details related to sub-gaussian random variables, we refer the readers to \cite{sg}.
			
			\begin{defn}
				A real random variable $X$ is called \emph{$b$-sub-gaussian} for some $b>0$ if
				\begin{equation}\label{sub gauss}
				\mathbb{E}(e^{tX})\leq e^{\frac{b^2t^2}{2}},\quad\forall t\in\mathbb{R}.
				\end{equation}
				Notation: $X\sim\sub(b^2)$.
			\end{defn}
			
			With the above definition, it's easy to see that if $X$ is a symmetric Bernoulli random variable, then $X\sim \sub(1)$.\\
			
			\vspace{0.2cm}
			
			It is straight forward to verify the following properties of sub-gaussian variables:

			\begin{prop}\label{propsg}
				\begin{enumerate}[(i)]
					
					\item If $X\sim\sub(b^2)$, then $X$ is centred. That is, $\mathbb{E}(X)=0$.
					
					\item If $X\sim\sub(b^2)$, then $\var(X)\leq b^2$.
					
					\item If $X\sim\sub(b^2)$, then for every $c\in\mathbb{R}$, we have $c X\sim\sub(c^2b^2)$.
					
					\item If $X_1\sim\sub(b_1^2)$ and $X_2\sim\sub(b_2^2)$, then $X_1+X_2\sim\sub((b_1+b_2)^2)$. Moreover, if $X_1$ and $X_2$ are independent, then $X_1+X_2\sim\sub(b_1^2+b_2^2)$.

					\item If $X=(X_1,X_2,\dots,X_N)$ is a random vector such that $X_i$'s are i.i.d. random variables with $X_i\sim\sub(b^2)$ for all $i=1,\dots,N$. Then for every $\alpha\in\mathbb{R}^N$, we have
					$$\langle X,\mathbf{\alpha}\rangle\sim\sub(\Vert\alpha\Vert_2^2b^2).$$

				\end{enumerate}
			\end{prop}
			
			Sub-gaussian random variables can be characterized in several equivalent ways:

			\begin{thm}\label{csg}The following statements are equivalent for a real centred random variable $X$:
				
				\begin{enumerate}[(i)]
					
					\item $X\sim\sub(b^2)$ for some $b>0$.
					
					\item For all $\lambda>0$, there exists $ c>0$ such that $\mathbb{P}(|X|>\lambda)\leq 2e^{-c\lambda^2}$. In fact we may choose $c=\frac{1}{2b^2}$.
					
					\item For every $ \xi>1$, there exists $a>0$ depending on $\xi$ such that $\mathbb{E}(e^{aX^2})\leq \xi$. In fact, we can choose $a=\frac{\xi-1}{2b^2(\xi+1)}$.
					
				\end{enumerate}
			\end{thm}
			
			\subsection{The Concentration of Measure Phenomena}
			The concentration of measure phenomena is extremely important in random matrix theory. It has been intensively studied over the past few years. For detailed discussions on this topic, we refer the readers to \cite{da,consg, lot, cmb,hdp,con, pbs} and references therein.\\

			\vspace{0.2cm}
			
			For us to establish the SRIP for Bernoulli random matrices, we need two concentration inequalities for sub-gaussian random variables. The first inequality is the classical concentration of measure inequality for sub-gaussian random matrices, which was proved in \cite{consg,cmb}:
			
			\begin{thm}[\textbf{Concentration of measure for sub-gaussian matrices}] \label{con} Suppose that $A\in\mathbb{R}^{n\times m}$ is a random matrix whose entries are drawn from i.i.d. $b$-sub-gaussian random variables for some $b>0$ and has variance $\frac{1}{n}$. Then for all $x\in\mathbb{R}^m$, we have:
				\begin{enumerate}[(i)]
					
					\item $\mathbb{E}(\| Ax\|_2^2)=\| x\|_2^2$.
					
					\item There exists $\kappa>0$ (depending on the sub-gaussian random variable used to generate $A$) such that for every $\epsilon\in(0,1)$, we have
					$$\mathbb{P}(\| Ax\|_2^2-\| x\|_2^2>\epsilon \| x\|_2^2)<\exp(-\kappa\epsilon^2n),$$
					$$\mathbb{P}(\| Ax\|_2^2-\| x\|_2^2<-\epsilon \| x\|_2^2)<\exp(-\kappa\epsilon^2n).$$
				\end{enumerate}
			\end{thm}
			
			\vspace{0.2cm}
			
			\begin{rem}\label{conber}As a special case of Theorem\autoref{con}, it was proved in \cite{da} that if $A\in\mathbb{R}^{n\times m}$ has entries drawn from i.i.d. standard normal or symmetric Bernoulli random variables, then we have
				$$\mathbb{P}\left(\frac{1}{n}\Vert Ax\Vert_2^2-\Vert x\Vert_2^2>\epsilon \Vert x\Vert_2^2\right)<\exp\left(-\left(\frac{\epsilon^2}{4}-\frac{\epsilon^3}{6}\right)n\right)$$
				and
				$$\mathbb{P}\left(\frac{1}{n}\Vert Ax\Vert_2^2-\Vert x\Vert_2^2<-\epsilon \Vert x\Vert_2^2\right)<\exp\left(-\left(\frac{\epsilon^2}{4}-\frac{\epsilon^3}{6}\right)n\right)$$
				for all $x\in\mathbb{R}^m$ and $\epsilon\in(0,1)$. Thus in this case we may take $\kappa=\frac{1}{12}$.
			\end{rem}

			\vspace{0.2cm}
			
			Another inequality we need is a Lipschitz type concentration inequality for sub-gaussian random variables. This inequality for Gaussian random variables was established in \cite[Proposition 2.3]{hdp}. With only a minor modification of the proof of the Gaussian case, we can extend the result to sub-gaussian random variables. For the sake of convenience, we provide the result and its proof here.
			
			\begin{thm}[\textbf{Lipschitz concentration inequality for sub-gaussian distributions}]\label{lcsg}Let $X=(X_1,\dots,X_d)$ has real i.i.d. entries with $X_i\sim\sub(b^2)$ for all $i=1,\dots,N$. Let $f:\mathbb{R}^d\to\mathbb{R}$ be a $1$-Lipschitz function. That is, $|f(x)-f(y)|\leq\|x-y\|_2$ for all $x,y\in\mathbb{R}^d$. Then for all $t>0$,  we have
			\begin{equation}\label{lcpos}\mathbb{P}(f(X)-\mathbb{E}(f(X))\geq t)\leq\exp\left(-\frac{t^2}{5b^2}\right),\end{equation}
			\begin{equation}\label{lcneg}\mathbb{P}(f(X)-\mathbb{E}(f(X))\leq -t)\leq\exp\left(-\frac{t^2}{5b^2}\right).\end{equation}
			\end{thm}
		
		\begin{proof}
			
	By Rademacher's theorem, a Lipschitz function is differentiable almost everywhere. So it suffices to prove the result for every differentiable function $f$, and the general case follows from a standard approximation argument. As $f$ has Lipschitz constant 1, we have $\Vert\nabla f\Vert_2\leq 1$. Without loss of generality, assume $\mathbb{E}(f(X))=0$. Let $X'=(X'_1,\dots,X'_d)$ be an independent copy of $X$. That is, $X'$ and $X$ have the same (joint) distribution and are indenepdent. Let $\gamma:[0,1]\to\mathbb{R}^d$ be a smooth path connecting $X$ and $X'$ with
	$$\gamma(t)=X'\cos \left(\frac{\pi}{2}t\right)+X\sin \left(\frac{\pi}{2}t\right),\quad\forall t\in[0,1].$$
	Then
	$$\gamma'(t)=\frac{\pi}{2}\left(-X'\sin \left(\frac{\pi}{2}t\right)+X\cos \left(\frac{\pi}{2}t\right)\right)=:\frac{\pi}{2}Y(t).$$
By item (v) of Proposition\autoref{propsg}, $Y$ has i.i.d. components with $Y_i\sim\sub(b^2)$ for all $i=1,\dots,d$. Moreover, by the fundamental theorem of line integral
	$$f(X)-f(X')=\frac{\pi}{2}\int_0^1\langle \nabla f(\gamma(t)), Y(t)\rangle\diff t.$$
	As $\|\nabla f\|_2\leq 1$, and $Y$ has i.i.d. $\sub(b^2)$ entries, it follows that 
	$$\langle \nabla f(\gamma(t)), Y(t)\rangle\sim\sub(b^2),\quad\forall t\in[0,1].$$ 
	Thus by Jensen's inequality, Fubini's theorem and item (iii) of Proposition\autoref{propsg}, we have
	$$\mathbb{E}(\exp(\lambda(f(X)-f(X'))))\leq\int_0^1\mathbb{E}(\exp(\frac{\pi}{2}\lambda\langle \nabla f(\gamma(t)), Y(t)\rangle))\diff t\leq \exp\left(\frac{b^2\lambda^2\pi^2}{8}\right)$$
	for all $\lambda\in\mathbb{R}$. As $X$ and $X'$ are i.i.d. copies, we have $\mathbb{E}(f(X))=\mathbb{E}(f(X'))=0$. Thus Jensen's inequality yields
	$$\mathbb{E}(\exp(\lambda f(X')))\geq\exp(\mathbb{E}(\lambda f(X')))=1,\quad\forall\lambda\in\mathbb{R}.$$

	Therefore
	$$\mathbb{E}(\exp(\lambda f(X)))\leq \mathbb{E}(\exp(\lambda(f(X)-f(X'))))\leq \exp\left(\frac{b^2\lambda^2\pi^2}{8}\right),\quad\forall \lambda\in\mathbb{R}.$$

	Thus for all $\lambda, t>0$,  we have
	$$\begin{aligned}\mathbb{P}(f(X)\geq t)&=\mathbb{P}(\exp(\lambda f(X))\geq e^{\lambda t})\leq \frac{\mathbb{E}(\exp(\lambda f(X)))}{e^{\lambda t}}\leq \exp\left(\frac{b^2\lambda^2\pi^2}{8}-\lambda t\right).
	\end{aligned}$$
	By setting $\lambda=\frac{4t}{b^2\pi^2}\left(1+\sqrt{1-\frac{\pi^2}{10}}\right)$, we have
	$$\mathbb{P}(f(X)\geq t)\leq\exp\left(-\frac{t^2}{5b^2}\right).$$
	This proves \eqref{lcpos}, and \eqref{lcneg} can be proved similarly. 
		\end{proof}

A direct consequence of Theorem\autoref{con} is the following corollary:

	\begin{cor}\label{key} Suppose that $y_1,\dots,y_n$ are i.i.d. random variables with $y_i\sim\sub(b^2)$ for $i=1,\dots,n$. Let $y_{(1)},\dots,y_{(n)}$ be the non-increasing rearrangements of $y_i$'s in magnitudes, i.e., $|y_{(1)}|\geq\dots\geq|y_{(n)}|$. Then
				$$\mathbb{E}\left(\sqrt{\frac{1}{k}\sum_{j=1}^ky_{(j)}^2}\right)\leq\sqrt{2eb^2\ln\frac{en}{k}},\quad \forall k=1,\dots,n.$$
			\end{cor}
			\begin{proof}Let $S\subseteq\{1,\dots,n\}$. Define
				$$F_S:\mathbb{R}^n\to\mathbb{R},\quad y\mapsto \sqrt{\sum_{j\in S}y_{(j)}^2}.$$
			Note that for $x,y\in\mathbb{R}^n$:
			$$\begin{aligned}|F_S(x)-F_S(y)|^2&=\sum_{j\in S}(x_{(j)}^2+y_{(j)}^2)-2\sqrt{\left(\sum_{j\in S}x_{(j)}^2\right)\left(\sum_{j\in S}y_{(j)}^2\right)}\\
			&\leq\sum_{j\in S}(x_{(j)}^2+y_{(j)}^2)-2\sum_{j\in S}|x_{(j)}y_{(j)}|\\
			&=\sum_{j\in S}(|x_{(j)}|-|y_{(j)}|)^2\\
			&\leq\sum_{j=1}^n(|x_{(j)}|-|y_{(j)}|)^2\\
			&=\sum_{j=1}^n(x_{(j)}^2+y_{(j)}^2)-2\sum_{j=1}^n|x_{(j)}y_{(j)}|\\
			&\leq\sum_{j=1}^n(x_j^2+y_j^2)-2\sum_{j=1}^n|x_jy_j|\quad(\text{by the rearrangement inequality})\\
			&\leq \|x-y\|_2^2.
			\end{aligned}$$
		Thus $F_S$ is $1$-Lipschitz. For all $t>0$, it follows from Theorem\autoref{con} that
				$$\mathbb{P}\left(\sqrt{\frac{1}{|S|}\sum_{j\in S}y_{(j)}^2}\geq t+\mathbb{E}\sqrt{\frac{1}{|S|}\sum_{j\in S}y_{(j)}^2}\right)\leq \exp\left(-\frac{t^2|S|}{5b^2}\right).$$

				Let $\xi>1$ and choose $t=\frac{\xi-1}{2b^2(\xi+1)}$. It follows that:
				$$\begin{aligned}\exp\left(\mathbb{E}\left(\frac{1}{k}\sum_{j=1}^kty_{(j)}^2\right)\right)&=\exp\left(\frac{1}{k}\sum_{j=1}^k\mathbb{E}(ty_{(j)}^2)\right)\leq \frac{1}{k}\sum_{j=1}^k\exp(\mathbb{E}(ty_{(j)}^2))\\
				&\leq\frac{1}{k}\sum_{j=1}^k\mathbb{E}(\exp(ty_{(j)}^2))\leq \frac{1}{k}\sum_{j=1}^n\mathbb{E}(\exp(ty_{(j)}^2))\\
				&=\frac{1}{k}\mathbb{E}\left(\sum_{j=1}^n\exp(ty_{j}^2)\right)=\frac{1}{k}\mathbb{E}\left(\sum_{j=1}^n\exp(ty_{j}^2)\right)\\
				&=\frac{1}{k}\sum_{j=1}^n\mathbb{E}(\exp(ty_{j}^2))\leq\frac{1}{k}\sum_{j=1}^n\xi \quad (\text{by (iii) of Theorem\autoref{csg}})\\
				&=\xi\frac{n}{k}.
				\end{aligned}$$
				By taking logarithms on both sides and using Jensen's inequality, we have
				$$ \mathbb{E}\sqrt{\frac{1}{k}\sum_{j=1}^ky_{(j)}^2}\leq \sqrt{\frac{2b^2(\xi+1)}{\xi-1}\ln\frac{\xi n}{k}}.$$
				Setting $\xi=e$ yields the result.\end{proof}
			
			\subsection{Linear Combinations of Symmetric Bernoulli Random Variables}
			
		To study the SRIP of Bernoulli random matrices, it is inevitable to encounter linear combinations of symmetric Bernoulli random variables. We first recall the famous Khinchine's inequality:
		
			\begin{thm}[\textbf{Khinchine's inequality}]Let $X_1,\dots,X_m$ i.i.d. symmetric Bernoulli random variables. Then for every $p\in(0,\infty)$, there exist $a(p),b(p)>0$ depending only on $p$ so that
				$$a(p)\|c\|_2\leq\left(\mathbb{E}\left|\sum_{j=1}^mc_jX_j\right|^p\right)^{\frac{1}{p}}\leq b(p)\|c\|_2,\quad\forall c=(c_1,\dots,c_m)\in\mathbb{C}^m.$$
				The optimal $a(p)$ and $b(p)$ are called the Khinchine's constants.
			\end{thm}
			The exact values of the Khinchine's constants were given by Haagerup in \cite{kh} as follows:
			\begin{thm}The optimal $a(p)$ and $b(p)$ in Khinchine's inequality are given by
				$$b(p)=\begin{cases}1, &0<p\leq 2,\\
				\sqrt{2}\left(\frac{\Gamma\left(\frac{p+1}{2}\right)}{\sqrt{\pi}}\right)^{\frac{1}{p}}, &p>2,\end{cases}$$
				and
				$$a(p)=\begin{cases}2^{\frac{1}{2}-\frac{1}{p}}, &0<p\leq p_0,\\
				\sqrt{2}\left(\frac{\Gamma\left(\frac{p+1}{2}\right)}{\sqrt{\pi}}\right)^{\frac{1}{p}}, &p_0<p<2,\\
				1, &p\geq 2,\end{cases}$$
				where $p_0\in(0,2)$ satisfies $\Gamma\left(\frac{p_0+1}{2}\right)=\frac{\sqrt{\pi}}{2}$.
			\end{thm}
			
			With Khinchine's inequality and constants, and use the idea as in the proof of \cite[Lemma 2.2]{us}, we can prove the following key lemma for Bernoulli random matrices.
			
			\begin{lem}\label{khb}Let $\Phi\in\mathbb{R}^{n\times m}$ be a Bernoulli random matrix. Then
				$$\mathbb{P}\left(\|\Phi x\|_2^2\leq\left(\frac{1}{2}-q\right)n\right)<\exp\left(-\frac{qn}{3}\right),\quad\forall x\in\mathbb{S}^{m-1},q\in\left(0,\frac{1}{2}\right).$$
			\end{lem}

		\begin{proof}Let $x_0\in \mathbb{S}^{m-1}$ and put $Y_0:=\|\Phi x_0\|_2^2$.	For every $\mu\in\mathbb{R}$, define 
			$$F(\mu)=\ln[\mathbb{E}\exp(-\mu Y_0)].$$ 
			By Markov's inequality, we have
			$$\begin{aligned}
			\mathbb{P}(-\mu Y_0\geq F(\mu)+\nu)&=\mathbb{P}(\exp(-\mu Y_0-F(\mu))\geq \exp(\nu))\\
			&\leq\frac{\mathbb{E}(\exp(-\mu Y_0-F(\mu))}{\exp(\nu)}\\
			&=\exp(-\nu)
			\end{aligned}$$
			for all $\nu\in\mathbb{R}$. Recall that 
			$$1-t\leq\exp(-t)\leq 1-t+\frac{t^2}{2},\quad\forall t\geq 0.$$
			Denote $\phi_{ij}$ the entry of $\Phi$ in the $i-$th row and $j-$th colume. It follows that
			$$\begin{aligned}
			\mathbb{E}\exp(-\mu Y_0)&=\prod_{i=1}^n\mathbb{E}\left[\exp\left(-\mu\left|\sum_{j=1}^m (x_0)_j\phi_{ij}\right|^2\right)\right]\\
			&\leq \prod_{i=1}^n\left[1-\mu\mathbb{E}\left|\sum_{j=1}^m (x_0)_j\phi_{ij}\right|^2+\frac{\mu^2}{2}\mathbb{E}\left|\sum_{i=1}^m (x_0)_j\phi_{ij}\right|^4\right]\\
			&\leq \prod_{i=1}^n\left[1-\mu+\frac{\mu^2}{2}b(4)^4\right]\\
			&\leq\exp\left(-n\mu+\frac{3}{2}n\mu^2\right),
			\end{aligned}$$
			whenever $\mu-\frac{3}{2}\mu^2\geq 0$, where we have used the exact value of the Khinchine's constant 
			$$b(4)=\sqrt{2}\left(\frac{\Gamma(5/2)}{\sqrt{\pi}}\right)^{\frac{1}{4}}=\sqrt{2}\left(\frac{3}{4}\right)^{\frac{1}{4}}.$$
			Therefore
			$$F(\mu)\leq -n\mu+\frac{3}{2}n\mu^2.$$
			It follows that
			$$\mathbb{P}\left(-\mu Y_0\geq -n\mu+\frac{3}{2}n\mu^2+\nu\right)\leq\exp(-\nu),\quad$$
			whenever $\mu-\frac{3}{2}\mu^2\geq 0$ and $\nu>0$. By setting $\mu=\frac{1}{3}$ and $\nu=\frac{qn}{3}$, the proof is complete.\end{proof}
			
			\section{Bernoulli Random Matrices under Arbitrary Erasure of Rows with a Given Portion of Corruption}

			In this section, we will study the erasure robustness property of Bernoulli random matrices. We will establish the SRIP and the robust version of Johnson-Lindenstrauss lemma for Bernoulli random matrices. Throughout this section:
			
			\begin{itemize}
				
				\item  $\Phi\in\mathbb{R}^{n\times m}$ is a Bernoulli random matrix. As we are studying the erasure robusness property of the matrix, we should assume that our matrix has at least two rows ($n\geq 2$). In applications, usually the dimensions of matrices are huge.

				\item For $T\subseteq\{1,\dots n\}$, $\Phi_T$ denotes the sub-matrix of $\Phi$ by keeping rows with indices in $T$.
				
				\item For i.i.d. real random variables $y_1,\dots,y_n$, we denote $y_{(1)},\dots,y_{(n)}$ the non-increasing rearrangements of $y_i$'s in magnitude. That is, $|y_{(1)}|\geq\dots\geq|y_{(n)}|$.

			\end{itemize}
			
			We introduce the following notations:
			
			\begin{itemize}
			
		\item Fix $x\in\mathbb{S}^{m-1}$. For $\beta\in[0,1)$ and $ 0\leq\theta\leq\omega\leq\infty$, define the following events:
			$$\Omega_{[\theta,\omega],\beta}=\left\{\frac{1}{|T|}\|\Phi_Tx\|_2^2\in[\theta,\omega],\text{ for all } T\subseteq\{1,\dots,n\}\text{ with }|T^c|\leq\beta n\right\},$$
			$$\tilde{\Omega}_{[\theta,\omega],\beta}=\left\{\frac{1}{n}\|\Phi_Tx\|_2^2\in[\theta,\omega],\text{ for all } T\subseteq\{1,\dots,n\}\text{ with }|T^c|\leq\beta n\right\}.$$

			\item For $\beta\in[0,1]$, define
			$$T_\beta=\{T\subseteq\{1,\dots,n\}: |T^c|=\lfloor\beta n\rfloor\}.$$
			
		\item Fix $x\in\mathbb{S}^{m-1}$. For $\beta\in(0,1)$ and $\alpha>0$, define
			$$\theta_\beta(\alpha)=\sup\{\theta\in[0,\infty]: \mathbb{P}(\Omega_{[\theta,\infty],\beta})\geq 1-e^{-\alpha n},\forall n\geq 2\},$$
			$$\omega_\beta(\alpha)=\inf\{\omega\in[0,\infty]: \mathbb{P}(\Omega_{[0,\omega],\beta})\geq 1-e^{-\alpha n},\forall n\geq 2\},$$
			$$\tilde{\theta}_\beta(\alpha)=\sup\{\theta\in[0,\infty]: \mathbb{P}(\tilde{\Omega}_{[\theta,\infty],\beta})\geq 1-e^{-\alpha n},\forall n\geq 2\},$$
			$$\tilde{\omega}_\beta(\alpha)=\inf\{\omega\in[0,\infty]: \mathbb{P}(\tilde{\Omega}_{[0,\omega],\beta})\geq 1-e^{-\alpha n},\forall n\geq 2\}.$$
		\end{itemize}

			\subsection{Auxiliary Results}
			
			The following simple observation was proved in \cite{rp}.
			
			\begin{lem}\label{ratio}Fix $x\in\mathbb{S}^{m-1}$. For $0\leq\gamma\leq\beta<1$, we have
				$$\min_{T\in T_\beta}\frac{1}{|T|}\|\Phi_Tx\|_2^2\leq \min_{T\in T_\gamma}\frac{1}{|T|}\|\Phi_Tx\|_2^2\leq \max_{T\in T_\gamma}\frac{1}{|T|}\|\Phi_Tx\|_2^2\leq \max_{T\in T_\beta}\frac{1}{|T|}\|\Phi_Tx\|_2^2,$$
				and
				$$\max_{T\in T_\gamma}\left|\frac{1}{|T|}\|\Phi_Tx\|_2^2-1\right|\leq\max_{T\in T_\beta}\left|\frac{1}{|T|}\|\Phi_Tx\|_2^2-1\right|.$$
				As a consequence, we have $\Omega_{[\theta,\omega],\beta}\subseteq \Omega_{[\theta,\omega],\gamma}$ for all $0\leq\theta\leq\omega\leq\infty$. 
			\end{lem}
			
		We now provide estimates for  $\theta_\beta(\alpha),\omega_\beta(\alpha),\tilde{\theta}_\beta(\alpha),\tilde{\omega}_\beta(\alpha)$. Later we will use these estimates to establish our main results.
		
			\begin{lem}[\textbf{lower estimates for $\theta_\beta(\alpha)$ and $\tilde{\theta}_\beta(\alpha)$}] \label{ltb}For $\beta\in (0,1)$, choose $q_\beta\in\left(0,\frac{1}{2}\right)$ such that 
				$$(1-\beta)\ln(1-\beta)+\beta\ln(\beta)+\frac{q_\beta}{3}(1-\beta)>0,$$
			and let $\alpha\in\left(0,\ln\left((1-\beta)^{1-\beta}\beta^{\beta}\right)+\frac{q_\beta}{3}(1-\beta)\right)$. Denote $t_{\pm 1}\approx 0.0376$ the unique root of 
			$$f(t)=(1-t)\ln(1-t)+t\ln(t)+\frac{1}{6}(1-t)$$
			on $(0,\frac{1}{2})$. Then for each $\beta\in(0,t_{\pm1})$, we have
				$$\theta_\beta(\alpha)\geq \min\left\{ \frac{1}{2}-3(1-\beta)^{-1}\left[\alpha-\ln\left((1-\beta)^{1-\beta}\beta^{\beta}\right)\right],\frac{1}{2}-3\alpha\right\},$$
				$$\tilde{\theta}_\beta(\alpha)\geq(1-\beta)\min\left\{\frac{1}{2}-3(1-\beta)^{-1}\left[\alpha-\ln\left((1-\beta)^{1-\beta}\beta^{\beta}\right)\right],\frac{1}{2}-3\alpha\right\}.$$
			
			\end{lem}
			
			\begin{proof}
				
				Fix $x\in\mathbb{S}^{m-1}$. For $\theta>0$ and $\beta\in\left(0,1\right)$, we have
				$$\Omega_{[\theta,\infty],\beta}=\left\{\frac{1}{|T|}\Vert \Phi_Tx\Vert_2^2\geq\theta,\quad \forall T\subseteq\{1,\dots,n\}, |T^c|\leq\beta n\right\}.$$
			Let $\alpha>0$. To find a lower estimate for $\theta_\beta(\alpha)$, we need to determine the values of $\theta$ which satisfy 
				$$\mathbb{P}(\Omega_{[\theta,\infty],\beta}^c)<\exp(-\alpha n).$$
				Let $|T^c|=k\leq\beta n$ and $\gamma=k/n$. By Lemma\autoref{khb}, we have
				$$\mathbb{P}\left(\| \Phi_Tx\|_2^2\leq \left(\frac{1}{2}-q\right)(n-k)\right)\leq\exp\left(-\frac{q}{3}(n-k)\right),\quad\forall q\in\left(0,\frac{1}{2}\right).$$
				Equivalently, we have
				\begin{equation}\label{sb3}
				\mathbb{P}\left(\frac{1}{|T|}\| \Phi_Tx\|_2^2\leq \frac{1}{2}-q\right)\leq\exp\left(-\frac{q}{3}(1-\gamma)n\right),\quad\forall q\in\left(0,\frac{1}{2}\right).
				\end{equation}
				
				Recall the Stirling's approximation:
				\begin{equation}\label{stir}
				\sqrt{2\pi}n^{n+\frac{1}{2}}e^{-n}\leq n!\leq en^{n+\frac{1}{2}}e^{-n},\quad\forall n\in\mathbb{N}.
				\end{equation}
				It follows that for $k=1,\dots,n-1$, we have
				$$\begin{aligned}
				\binom{n}{k}&=\frac{n!}{(n-k)!k!}\leq\frac{e}{2\pi}\left(\frac{n}{n-k}\right)^{n-k}\left(\frac{n}{k}\right)^{k}\left(\frac{n}{(n-k)k}\right)^\frac{1}{2}\leq\frac{e}{\sqrt{2}\pi}\left(\frac{n}{n-k}\right)^{n-k}\left(\frac{n}{k}\right)^{k},
				\end{aligned}$$
				where in the last inequality we used the fact that $\frac{n}{(n-k)k}\leq2$ whenever $n\geq 2$ and  $1\leq k\leq n-1$. Now consider two cases:
				\begin{itemize}
					
					\item If $\lfloor\beta n\rfloor=0$, by Lemma\autoref{khb}:
					$$\mathbb{P}(\Omega_{[\theta,\infty],\beta}^c)=\mathbb{P}\left(\| \Phi x\|_2^2<\theta n\right)\leq \exp\left(-\frac{1}{3}\left(\frac{1}{2}-\theta\right)n\right),\quad\forall\theta\in\left(0,\frac{1}{2}\right).$$ 
					By letting $\theta\leq \frac{1}{2}-3\alpha$ with $\alpha\in(0,\frac{1}{6})$, we have $\mathbb{P}(\Omega_{[\theta,\infty],\beta}^c)<\exp(-\alpha n)$.
					
					\item If $\lfloor\beta n\rfloor\geq 1$, note that
					\begin{equation}\label{omth}\mathbb{P}(\Omega_{[\theta,\infty],\beta}^c)=\mathbb{P}\left(\min_{|T^c|\leq\beta n}\frac{1}{|T|}\| \Phi_Tx\|_2^2< \theta\right)=\mathbb{P}\left(\min_{T\in T_{\beta_n}}\frac{1}{|T|}\| \Phi_Tx\|_2^2< \theta\right),\end{equation}
					where the last equality follows from Lemma\autoref{ratio} and $\beta_n=\frac{\lfloor \beta n\rfloor}{n}$. Let $k_n:=\beta_nn$, we have
					$$\begin{aligned}
					\binom{n}{k_n}&\leq \left(\frac{e}{\sqrt{2}\pi}\right)\left(\frac{1}{1-\beta_n}\right)^{(1-\beta_n)n}\left(\frac{1}{\beta_n}\right)^{\beta_n n}=\left(\frac{e}{\sqrt{2}\pi}\right)\exp\left(n\ln\left((1-\beta_n)^{-(1-\beta_n)}\beta_n^{-\beta_n}\right)\right).
					\end{aligned}$$
					As \eqref{sb3} holds when $|T^c|=k_n$ and $\gamma=\beta_n$, it follows that for every $q\in\left(0,\frac{1}{2}\right)$:
				\begin{equation}\label{unbd}\begin{aligned}\mathbb{P}\left(\min_{T\in T_{\beta_n}}\frac{1}{|T|}\| \Phi_Tx\|_2^2\leq \frac{1}{2}-q\right)\leq&\binom{n}{k_n}\exp\left(-\frac{q}{3}(1-\beta_n)n\right)\\
					\leq &\left(\frac{e}{\sqrt{2}\pi}\right)\exp\left[-n\ln\left((1-\beta_n)^{(1-\beta_n)}\beta_n^{\beta_n}\right)-\frac{q}{3}(1-\beta_n)n\right].
					\end{aligned}\end{equation}
					Now we want to bound $\mathbb{P}(\Omega_{[\theta,\infty],\beta}^c)$ from above by $C\exp(-\alpha n)$ where $C$ is a positive constant independent of $n$. From what we have so far, it suffices to have $C=\frac{e}{\sqrt{2}\pi}$ and
					\begin{equation}\label{sb4}
					(1-\beta_n)\ln(1-\beta_n)+\beta_n\ln(\beta_n)+\frac{q}{3}(1-\beta_n)\geq \alpha.\end{equation}
					For every $c\in\left(0,\frac{1}{2}\right]$, define
					\begin{equation}\label{ft}
					f_c(t)=(1-t)\ln(1-t)+t\ln(t)+\frac{c}{3}(1-t),\quad t\in(0,1).
					\end{equation}
					Then 
					$$f_c'(t)=\ln(t)-\ln(1-t)-\frac{c}{3},\quad\forall t\in(0,1).$$
					It is easy to see that $f'_c(t)<0$ whenever $0<t<\frac{\exp\left(\frac{c}{3}\right)}{1+\exp\left(\frac{c}{3}\right)}$. Since $\frac{\exp\left(\frac{c}{3}\right)}{1+\exp\left(\frac{c}{3}\right)}\geq\frac{1}{2}$ whenever $0<c\leq\frac{1}{2}$, it follows that $f_c(t)$ is strictly decreasing on $\left(0,\frac{1}{2}\right)$, and
					$$f_c\left(\frac{1}{2}\right)=\ln\frac{1}{2}+\frac{c}{6}\leq -\ln 2+\frac{1}{12}<0.$$
					On the other hand, note that
					$$\lim_{t\to 0^+}f_c(t)=\frac{c}{3}>0.$$
					It follows that $f_c$ has a unique root on $\left(0,\frac{1}{2}\right)$, call this root $t_c$. Thus we have
					$$(1-t_c)\ln(1-t_c)+t_c\ln(t_c)+\frac{c}{3}(1-t_c)=0.$$
					Note that for $c'\geq c$, we have
					$$(1-t_c)\ln(1-t_c)+t_c\ln(t_c)+\frac{c'}{3}(1-t_c)\geq 0.$$
				Since $f_{c'}(t)$ is strictly decreasing on $\left(0,\frac{1}{2}\right)$, it follows that $t_{c'}\geq t_c$. That is, $t_c$ increases as $c$ increases. Define
					$$t_{\pm 1}:=t_{\frac{1}{2}}.$$
				It follows that
					$$(1-t_{\pm 1})\ln(1-t_{\pm 1})+t_{\pm 1}\ln(t_{\pm 1})+\frac{1}{6}(1-t_{\pm 1})=0.$$
				Whenever $\beta\in(0,t_{\pm 1})$, it's not difficult to see that there exists $c_{\beta}\in\left(0,\frac{1}{2}\right)$ such that
					$$f_{c_{\beta}}(\beta)=(1-\beta)\ln(1-\beta)+\beta\ln(\beta)+\frac{c_{\beta}}{3}(1-\beta)=0.$$
			Indeed, $c_{\beta}=-3(\ln(1-\beta)+\frac{\beta}{1-\beta}\ln(\beta))$, see Figure \ref{cbeta} for an illustration.	Choose $q_\beta\in(c_{\beta},\frac{1}{2})$, we have	
			$$\begin{aligned}&(1-\beta_n)\ln(1-\beta_n)+\beta_n\ln(\beta_n)+\frac{q_\beta}{3}(1-\beta_n)\\
			\geq&(1-\beta)\ln(1-\beta)+\beta\ln(\beta)+\frac{q_\beta}{3}(1-\beta)\\
			>&0.\end{aligned}$$
			
			\begin{figure}
				\centering
				\includegraphics[width=0.5\linewidth]{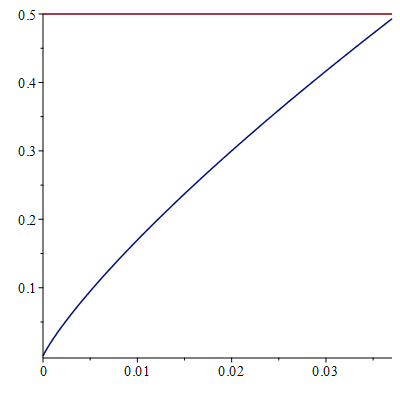}
				\caption{The graph above is a plot of $c_{\beta}$ when $\beta\in(0,t_{\pm1})$. The blue curve is the graph of $c_{\beta}$ as a function of $\beta$. The horizontal red line is the graph for the constant function which attains the value $\frac{1}{2}$ at every point. The figure illustrates that $c_{\beta}<\frac{1}{2}$ whenever $\beta\in(0,t_{\pm 1})$}
				\label{cbeta}
			\end{figure}

			Now for $\alpha\in\left(0,(1-\beta)\ln(1-\beta)+\beta\ln(\beta)+\frac{q_\beta}{3}(1-\beta)\right)$, we have
					$$q_\beta\geq 3(1-\beta_n)^{-1}\left[\alpha-\ln\left((1-\beta_n)^{1-\beta_n}\beta_n^{\beta_n}\right)\right].$$
					Therefore, if $\theta\leq \frac{1}{2}-3(1-\beta)^{-1}\left[\alpha-\ln\left((1-\beta)^{1-\beta}\beta^{\beta}\right)\right]$, it follows from \eqref{omth} and \eqref{unbd} that
					$$\mathbb{P}(\Omega_{[\theta,\infty],\beta}^c)\leq\frac{e}{\sqrt{2}\pi}\exp(-\alpha n).$$
					
				\end{itemize}
				
			Consequently, we have
				$$\alpha<\min\left\{\ln\left((1-\beta)^{1-\beta}\beta^{\beta}\right)+\frac{q_{\beta}}{3}(1-\beta),\frac{1}{6}\right\},$$
				and
				$$\theta\leq \min\left\{ \frac{1}{2}-(1-\beta)^{-1}\left[\alpha-\ln\left((1-\beta)^{1-\beta}\beta^{\beta}\right)\right],\frac{1}{2}-3\alpha\right\}.$$
				As $f_{q_\beta}$ is decreasing on $(0,\frac{1}{2})$ and $\lim_{t\to 0^+}f_{q_\beta}(t)=\frac{q_\beta}{3}<\frac{1}{6}$, it follows that
				$$\ln\left((1-\gamma)^{1-\gamma}\gamma^{\gamma}\right)+\frac{q_{\beta}}{3}(1-\gamma)<\frac{1}{6},\quad \forall0<\gamma<t_{\pm 1}<\frac{1}{2}.$$
				Thus $\alpha<\ln\left((1-\beta)^{1-\beta}\beta^{\beta}\right)+\frac{q_{\beta}}{3}(1-\beta)$ whenever $\beta\in(0, t_{\pm 1})$. This finishes the proof of the lower estimate for $\theta_\beta(\alpha)$.\\
				
				\vspace{0.2cm}
				
				To obtain the lower estimate for $\tilde{\theta}_\beta(\alpha)$, observe that
				$$\frac{1}{n}\sum_{j=k_n+1}^ny_{(j)}^2=\frac{n-k_n}{n}\frac{1}{n-k_n}\sum_{j=k_n+1}^ny_{(j)}^2=\frac{1-\beta_n}{n-k_n}\sum_{j=k_n+1}^ny_{(j)}^2,$$
			where 
			$$y:=[y_1,\dots,y_n]^T=\Phi x.$$
			It follows that for $\theta> 0$:
				$$\begin{aligned}\mathbb{P}(\tilde{\Omega}_{[\theta,\infty],\beta})&=\mathbb{P}\left(\min_{T\in T_\beta}\frac{1}{n}\Vert \Phi_Tx\Vert_2^2\geq\theta\right)=\mathbb{P}\left(\frac{1}{n}\sum_{j=k_n+1}^ny_{(j)}^2\geq\theta\right)\\
				&=\mathbb{P}\left(\sqrt{\frac{1}{n-k_n}\sum_{j=k_n+1}^ny_{(j)}^2}\geq\sqrt{\frac{\theta}{1-\beta_n}}\right)=\mathbb{P}(\Omega_{\left[\frac{\theta}{1-\beta_n},\infty\right],\beta})\geq\mathbb{P}(\Omega_{\left[\frac{\theta}{1-\beta},\infty\right],\beta}).
				\end{aligned}$$
				Thus $\tilde{\theta}_\beta(\alpha)\geq (1-\beta)\theta_\beta(\alpha)$ and the lower estimate of $\tilde{\theta}_\beta(\alpha)$ follows immediately. \end{proof}
			
			The upper estimates for $\theta_\beta(\alpha)$ and $\tilde{\theta}_\beta(\alpha)$ can be deduced by applying the same argument  as in the proof of \cite[Lemma 4.1]{rp} with only a little modification. 
			
			\begin{lem}[\textbf{upper estimates for $\theta_\beta(\alpha)$ and $\tilde{\theta}_\beta(\alpha)$}]\label{ut} For $\beta\in (0,1)$ and $\alpha\in(0,1)$, we have $\theta_\beta(\alpha)\leq 1$ and $\tilde{\theta}_\beta(\alpha)\leq 1-\beta$.	
			\end{lem}
			\begin{proof}Fix $x\in\mathbb{S}^{m-1}$, let $y=Ax$. Set $\beta_n=\frac{\lfloor\beta n\rfloor}{n}$ and $k_n:=\beta_nn$. For every $\theta>0$, we have
				$$\mathbb{P}(\Omega_{[\theta,\infty],\beta})=\mathbb{P}\left(\sqrt{\frac{1}{n-k_n}\sum_{j=k_n+1}^n y_{(j)}^2}\geq\sqrt{\theta}\right).$$
			Note that
				$$\begin{aligned}
				\mathbb{E}\sqrt{\frac{1}{n-k_n}\sum_{j=k_n+1}^n y_{(j)}^2}&\leq\sqrt{\frac{1}{n-k_n}\sum_{j=k_n+1}^n \mathbb{E}y_{(j)}^2}\leq\sqrt{\frac{1}{n}\sum_{j=1}^n \mathbb{E}y_{(j)}^2}=1.
				\end{aligned}$$
			If $\delta=\sqrt{\theta}-1>0$, we have
				$$\begin{aligned}
				\mathbb{P}(\Omega_{[\theta,\infty],\beta})&=\mathbb{P}\left(\sqrt{\frac{1}{n-k_n}\sum_{j=k_n+1}^n y_{(j)}^2}-\mathbb{E}\sqrt{\frac{1}{n-k_n}\sum_{j=k_n+1}^n y_{(j)}^2}\geq\sqrt{\theta}-\mathbb{E}\sqrt{\frac{1}{n-k_n}\sum_{j=k_n+1}^n y_{(j)}^2}\right)\\
				&\leq \mathbb{P}\left(\sqrt{\frac{1}{n-k_n}\sum_{j=k_n+1}^n y_{(j)}^2}-\mathbb{E}\sqrt{\frac{1}{n-k_n}\sum_{j=k_n+1}^n y_{(j)}^2}\geq\sqrt{\theta}-1\right)\\
				&\leq\exp\left(-\frac{\delta^2}{5}(n-k_n)\right)\quad (\text{by Theorem\autoref{lcsg})}\\
				&=\exp\left(-\frac{\delta^2}{5}(1-\beta_n)n\right).
				\end{aligned}$$
		If $\mathbb{P}(\Omega_{[\theta,\infty],\beta})\geq 1-e^{-\alpha n}$, we will have $1-e^{-\alpha n}\leq \exp\left(-\frac{\delta^2}{5}(1-\beta_n)n\right)$ for all $n\geq 2$. By letting $n\to\infty$, we have $1\leq 0$, which is impossible. Hence we must have $\delta<0$, i.e. $\theta\leq 1$, and thus $\theta_\beta(\alpha)\leq 1$. Finally, $\tilde{\theta}_\beta(\alpha)\leq 1-\beta$ follows immediately from
		$$\mathbb{P}(\tilde{\Omega}_{[\theta,\infty],\beta})=\mathbb{P}(\Omega_{\left[\frac{\theta}{1-\beta_n},\infty\right],\beta}),\quad \theta>0.$$
	\end{proof}

			Next, we estimate $\omega_\beta(\alpha)$ and $\tilde{\omega}_\beta(\alpha)$. The following result is a generalization of the corresponding result for the Gaussian case. One can verify that the result actually holds as long as $\Phi$ has entries drawn from i.i.d. sub-gaussian random variables. 
			
			\begin{thm}[\textbf{Estimates of $\omega_\beta(\alpha)$ and  $\tilde{\omega}_\beta(\alpha)$}]\label{eo} For $\beta\in(0,1)$ and $\alpha>0$, we have
				$$\left(\sqrt{\frac{5\alpha}{1-\beta}}+\sqrt{2e\ln\frac{e}{1-\beta}}\right)^2\geq \omega_\beta(\alpha)\geq\begin{cases}1, &\beta n<1\\
				\frac{1}{1-\frac{\beta}{2}}, &\beta n\geq 1,\end{cases}$$
				$$1\leq \tilde{\omega}_\beta(\alpha)\leq 1+\sqrt{12\alpha}.$$
			\end{thm}	
			\begin{proof}Fix $x\in\mathbb{S}^{m-1}$, let $y=\Phi x$. Set $\beta_n=\frac{\lfloor\beta n\rfloor}{n}$ and $k_n:=\beta_nn$.\\
				
				\vspace{0.2cm}
				
				First we estimate $\tilde{\omega}_\beta(\alpha)$. For $\omega>0$, by definition:
				$$\tilde{\Omega}_{[0,\omega],\beta}=\left\{\max_{T\in T_\beta}\frac{1}{n}\Vert \Phi_Tx\Vert_2^2\leq\omega\right\}=\left\{\frac{1}{n}\Vert \Phi x\Vert_2^2\leq\omega\right\}=\tilde{\Omega}_{[0,\omega],0}.$$
				We claim that $\tilde{\omega}_\beta(\alpha)\geq 1$. Assume not, then there exists $\omega<1$ such that 
				$$\mathbb{P}(\tilde{\Omega}_{[0,\omega],\beta})>1-\exp(-\alpha n).$$
				On the other hand, from the concentration inequality for Bernoulli random matrices (Theorem\autoref{con} and Remark\autoref{conber}), we have
				$$\mathbb{P}(\tilde{\Omega}_{[0,\omega],\beta})<\exp\left(-\frac{1}{12}(1-\omega)^2 n\right).$$
				Hence we have $1-\exp(-\alpha n)<\exp\left(-\frac{1}{12}(1-\omega)^2 n\right)$ for all $n\geq 2$. By letting $n\to\infty $, we have $1<0$, which is a contradiction. Therefore $\tilde{\omega}_\beta(\alpha)\geq 1$.\\
				
				\vspace{0.2cm}
				
				For the upper estimate for $\tilde{\omega}_\beta(\alpha)$:
				$$\mathbb{P}\left(\frac{1}{n}\Vert \Phi x\Vert_2^2\leq1+\varepsilon\right)\geq 1-\exp\left(-\frac{1}{12}\varepsilon^2 n\right)=1-\exp(-\alpha n),$$
				where $\varepsilon=\sqrt{12\alpha}$. Thus $\tilde{\omega}_\beta(\alpha)\leq 1+\sqrt{12\alpha}.$ This finishes the estimation for $\tilde{\omega}_\beta(\alpha)$.\\
				
				\vspace{0.2cm}
				
				Next, we estimate $\omega_\beta(\alpha)$. Note that for every $\omega\geq 0$, we have
				$$\mathbb{P}(\Omega_{[0,\omega],\beta})=\mathbb{P}\left(\sqrt{\frac{1}{n-k_n}\sum_{j=1}^{n-k_n}y_{(j)}^2}\leq\sqrt{\omega}\right).$$
				Since a Bernoulli random variable is sub-gaussian with $b=1$, Corollary\autoref{key} yields
				$$\mathbb{E}\sqrt{\frac{1}{n-k_n}\sum_{j=1}^{n-k_n}y_{(j)}^2}\leq\sqrt{2e\ln\frac{en}{n-k_n}}=\sqrt{2e\ln\frac{e}{1-\beta_n}}.$$
				Therefore
				$$\begin{aligned}
				\mathbb{P}(\Omega_{[0,\omega],\beta})&=\mathbb{P}\left(\sqrt{\frac{1}{n-k_n}\sum_{j=1}^{n-k_n}y_{(j)}^2}-\mathbb{E}\sqrt{\frac{1}{n-k_n}\sum_{j=1}^{n-k_n}y_{(j)}^2}\leq\sqrt{\omega}-\mathbb{E}\sqrt{\frac{1}{n-k_n}\sum_{j=1}^{n-k_n}y_{(j)}^2}\right)\\
				&\geq \mathbb{P}\left(\sqrt{\frac{1}{n-k_n}\sum_{j=1}^{n-k_n}y_{(j)}^2}-\mathbb{E}\sqrt{\frac{1}{n-k_n}\sum_{j=1}^{n-k_n}y_{(j)}^2}\leq\sqrt{\omega}-\sqrt{2e\ln\frac{e}{1-\beta_n}}\right)\\
				&\geq1-\exp\left(-\frac{\delta^2}{5}(n-k_n)\right)\\
				&\geq 1-\exp(-\alpha n),
				\end{aligned}$$
				provided that
				$$\delta:=\sqrt{\omega}-\sqrt{2e\ln\frac{e}{1-\beta_n}}\geq\sqrt{\frac{5\alpha}{1-\beta_n}}>0.$$
				By $0\leq\beta_n\leq\beta<1$, it follows that $\mathbb{P}(\Omega_{[0,\omega],\beta})\geq 1-\exp(-\alpha n)$ if
				$$\sqrt{\omega}\geq\sqrt{\frac{5\alpha}{1-\beta}}+\sqrt{2e\ln\frac{e}{1-\beta}}.$$
				Therefore
				$$\omega_\beta(\alpha)\leq\left(\sqrt{\frac{5\alpha}{1-\beta}}+\sqrt{2e\ln\frac{e}{1-\beta}}\right)^2.$$
				This proves the upper estimate for $\omega_\beta(\alpha)$. For the lower estimate of $\omega_\beta(\alpha)$, note that if $k_n=0$, then we have $\Omega_{[0,\omega],\beta}=\tilde{\Omega}_{[0,\omega],\beta}$, so $\omega_\beta(\alpha)\geq 1.$ If $k_n>0$, then $\beta_n>\frac{\beta}{2}$ and 
				$$\frac{1}{n-k_n}\sum_{j=1}^{k_n}y_{(j)}^2=\frac{n}{n-k_n}\frac{1}{n}\sum_{j=1}^{k_n}y_{(j)}^2=\frac{1}{1-\beta_n}\frac{1}{n}\sum_{j=1}^{k_n}y_{(j)}^2.$$
				It follows that for $\omega\geq 0$:
				$$\begin{aligned}\mathbb{P}(\Omega_{[0,\omega],\beta})&=\mathbb{P}\left(\min_{T\in T_{\beta_n}}\frac{1}{|T|}\Vert \Phi_Tx\Vert_2^2\leq\omega\right)\\
				&=\mathbb{P}\left(\frac{1}{n-k_n}\sum_{j=1}^{k_n}y_{(j)}^2\leq\omega\right)\\
				&=\mathbb{P}\left(\sqrt{\frac{1}{n}\sum_{j=1}^{k_n}y_{(j)}^2}\leq\sqrt{\omega(1-\beta_n)}\right)\\
				&=\mathbb{P}(\tilde{\Omega}_{\left[0,\omega(1-\beta_n)\right],\beta})\\
				&\leq\mathbb{P}(\tilde{\Omega}_{\left[0,\omega\left(1-\frac{\beta}{2}\right)\right],\beta}).
				\end{aligned}$$
				Hence
				$$\begin{aligned}
				\omega_\beta(\alpha)&=\inf\{\omega: \mathbb{P}(\Omega_{[0,\omega],\beta})>1-\exp(-\alpha n)\}\\
				&\geq \inf\{\omega: \mathbb{P}(\tilde{\Omega}_{\left[0,\omega\left(1-\frac{\beta}{2}\right)\right],\beta})>1-\exp(-\alpha n)\}\\
				&=\frac{\tilde{\omega}_\beta(\alpha)}{1-\frac{\beta}{2}}\\
				&\geq\frac{1}{1-\frac{\beta}{2}}.
				\end{aligned}$$
				This finishes the estimation for $\omega_\beta(\alpha)$, and the proof is now complete.\end{proof}

			The last supporting result we need is the following well-known result (see e.g. \cite{CA}) on approximating the unit sphere with its finite sets.
			
				\begin{lem}\label{net}Let $S\subseteq\{1,\dots,m\}$ with $|S|=s$. Set
					$$\mathbb{S}^{m-1}_S=\{x\in \mathbb{S}^{m-1}: \supp(x)\subseteq S\}.$$
					Then for any $\epsilon>0$ there exists an $\epsilon$-net $Q_{S,\epsilon}\subseteq \mathbb{S}^{m-1}_S$ satisfying
					\begin{itemize}
						
						\item $\mathbb{S}^{m-1}_S\subseteq\bigcup_{q\in Q_{S,\epsilon}}B_{\frac{\epsilon}{8}}(q)$, where $B_{\frac{\epsilon}{8}}(q)=\{v\in\mathbb{R}^m:\|v-q\|_2<\frac{\epsilon}{8}\}$ .
						
						\item $|Q_{S,\epsilon}|\leq\left(\frac{24}{\epsilon}\right)^s$.
						
					\end{itemize}
				\end{lem}
			
			\subsection{Main Results}
			
			With the auxiliary results derived in the previous subsection, we are at the stage to prove the SRIP of Bernoulli random matrices. In order to make the statement of the main theorem simple, we fix the following notations:
			
			\begin{itemize}
				\item For each $\beta\in(0,1)$, set
				$$\alpha_\beta(t):=(1-\beta)\ln(1-\beta)+\beta\ln(\beta)+\frac{t}{3}(1-\beta)$$
				for all $t\in(0,1)$.

			\item $t_{\pm 1}\approx 0.0376$ is the unique zero of the function 
			$$f(t)=(1-t)\ln(1-t)+t\ln(t)+\frac{1}{6}(1-t)$$
			on $(0,\frac{1}{2})$.
			\end{itemize}

			\begin{thm}[\textbf{the strong restricted isometry property of Bernoulli random matrices}]\label{sripb} Let $\beta\in(0,t_{\pm 1})$ and choose $q_\beta\in(0,\frac{1}{2})$ such that $\alpha_\beta(q_\beta)>0$. Let $s,m,n\in\mathbb{N}$, $\alpha\in(0,\min\{\alpha_\beta(q_\beta),\frac{1}{12}\})$ and $\epsilon\in(0,1)$ be such that
			$$n>\alpha^{-1}\left(s\ln\left(\frac{24em}{\epsilon s}\right)+\ln 2\right).$$ 
			Let $\Phi$ be an $n\times m$ Bernoulli random matrix. Then
				 
				$$\begin{aligned}
				&\mathbb{P}\left(\tilde{\theta}_{\epsilon,\beta}\|u\|_2^2\leq\frac{1}{n}\|\Phi_T u\|_2^2\leq\tilde{\omega}_{\epsilon,\beta} \| u\|_2^2,\quad \forall\|u\|_0\leq s, |T^c|\leq\beta n\right)\\
				\geq &1-2\left(\frac{24em}{\epsilon s}\right)^s\exp(-\alpha n),
				\end{aligned}$$
				and
				$$\begin{aligned}
				&\mathbb{P}\left(\theta_{\epsilon,\beta}\| u\|_2^2\leq\frac{1}{|T|}\| \Phi_T u\|_2^2\leq \omega_{\epsilon,\beta}\| u\|_2^2,\quad \forall\|u\|_0\leq s, |T^c|\leq\beta n\right)\\
				\geq &1-2\left(\frac{24em}{\epsilon s}\right)^s\exp(-\alpha n),
				\end{aligned}$$
				where $\tilde{\theta}_{\epsilon,\beta}, \tilde{\omega}_{\epsilon,\beta},\theta_{\epsilon,\beta},\omega_{\epsilon,\beta}$ are positive constants which only depend on $\beta$ and $\epsilon$.

			\end{thm}

			\begin{proof}The proof is based on the the ideas from the proofs of [\cite{srip}, Theorem 2.1] and [\cite{sprip}, Lemma 5.1].
			
				 Let $T\subseteq\{1,\dots,n\}$ be such that $|T^c|\leq\beta n$. Define
				 \begin{equation}\label{ttilde}\tilde{\theta}=(1-\beta)\min\left\{ \frac{1}{2}-3(1-\beta)^{-1}\left[\alpha-\ln\left((1-\beta)^{1-\beta}\beta^{\beta}\right)\right],\frac{1}{2}-3\alpha\right\},\end{equation}
				 \begin{equation}\label{otilde}\tilde{\omega}=1+\sqrt{12\alpha},\end{equation}
				 \begin{equation}\label{o}\omega=\left(\sqrt{\frac{5\alpha}{1-\beta}}+\sqrt{2e\ln\frac{e}{1-\beta}}\right)^2.\end{equation}
				 
				  By Lemma\autoref{ltb} and \autoref{eo}, with probability at least $1-2e^{-\alpha n}$, we have
				$$\sqrt{\tilde{\theta}(1-\epsilon)}\| u\|_2\leq\frac{1}{\sqrt{n}}\Vert \Phi_T u\Vert_2\leq\sqrt{\tilde{\omega}(1+\epsilon)}\|u\|_2$$
				for any $\epsilon\in(0,1)$. Let $S$, $ \mathbb{S}^{m-1}_S$ and $Q_{S,\epsilon}$ be the same as in Lemma\autoref{net}. Define
				$$d=\sup\left\{\frac{1}{\sqrt{n}}\| \Phi_Tu\|_2: u\in\mathbb{S}^{m-1}_S, |T^c|\leq\beta n\right\}.$$
				For every $u\in \mathbb{S}^{m-1}_S$, there exists $v_u\in Q_{S,\epsilon}$ such that $\| u-v_u\|_2\leq\frac{\epsilon}{8}$. Hence
				$$\frac{1}{\sqrt{n}}\| \Phi_Tu\|_2\leq \frac{1}{\sqrt{n}}\| \Phi_Tv_u\|_2+\frac{1}{\sqrt{n}}\| \Phi_T(u-v_u)\Vert_2\leq \sqrt{\tilde{\omega}(1+\epsilon)}+\frac{d\epsilon}{8}.$$
				By the definition of $d$, we have $d\leq\sqrt{\tilde{\omega}(1+\epsilon)}+\frac{d\epsilon}{8}$, which implies that $d\leq\sqrt{\tilde{\omega}(1+2\epsilon)}.$
				
				On the other hand, Lemma\autoref{ltb} and \autoref{eo} yield
				$$\begin{aligned}\frac{1}{\sqrt{n}}\| \Phi_Tu\|_2&\geq\frac{1}{\sqrt{n}}\| \Phi_Tv_u\|_2-\frac{1}{\sqrt{n}}\| \Phi_T(u-v_u)\|_2\geq \sqrt{\tilde{\theta}}-\frac{\epsilon}{8}\sqrt{\tilde{\omega}}.
				\end{aligned}$$
				Choose $\epsilon>0$ small enough such that $\sqrt{\tilde{\theta}}-\frac{\epsilon}{8}\sqrt{\tilde{\omega}}>0$. Set
				$$\tilde{\theta}_{\epsilon,\beta}:=\left(\sqrt{\tilde{\theta}}-\frac{\epsilon}{8}\sqrt{\tilde{\omega}}\right)^2,\quad\tilde{\omega}_{\epsilon,\beta}=\tilde{\omega}(1+2\epsilon),$$
				$$\theta_{\epsilon,\beta}=(1-\beta)^{-1}\tilde{\theta}_{\epsilon,\beta},\quad\omega_{\epsilon,\beta}=\omega(1+2\epsilon).$$
				One can see that with probability at least $1-2\left(\frac{24}{\epsilon}\right)^s\exp(-\alpha n)$, we have
				$$\tilde{\theta}_{\epsilon,\beta}\| u\|_2^2\leq\frac{1}{n}\| A_T u\|_2\leq\tilde{\omega}_{\epsilon,\beta}\| u\|_2^2,\quad \forall u\in\mathbb{R}^m\text{ with }\supp(u)\subseteq S,\quad\forall T\subseteq\{1,\dots,n\}\text{ with }|T^c|\leq\beta n.$$
				As there are $\binom{m}{s}$ subsets of $\{1,\dots,m\}$ with cardinality $s$, then by the union bound argument and Stirling's approximation $\binom{m}{s}\leq\left(\frac{em}{s}\right)^s$, we see that
				$$\begin{aligned}
				&\mathbb{P}\left(\tilde{\theta}_{\epsilon,\beta}\| u\|_2^2\leq\frac{1}{n}\| \Phi_T u\|_2^2\leq\tilde{\omega}_{\epsilon,\beta}\| u\|_2^2,\quad \forall\|u\|_0\leq s, |T^c|\leq\beta n\right)\\
				\geq &1-2\left(\frac{24em}{\epsilon s}\right)^s\exp(-\alpha n),
				\end{aligned}$$
				and similarly
					$$\begin{aligned}
					&\mathbb{P}\left(\theta_{\epsilon,\beta}\| u\|_2^2\leq\frac{1}{|T|}\|\Phi_T u\|_2^2\leq\omega_{\epsilon,\beta}\| u\|_2^2,\quad \forall\|u\|_0\leq s, |T^c|\leq\beta n\right)\\
					\geq &1-2\left(\frac{24em}{\epsilon s}\right)^s\exp(-\alpha n),
					\end{aligned}$$
				provided that $s\ln\frac{24em}{\epsilon s}<\alpha n-\ln 2$. Thus the proof is complete.\end{proof}

			Another result which we can establish is the following robust version of the Johnson-Lindenstrauss lemma:
			
			\begin{thm}[\textbf{robust Johnson-Lindenstrauss lemma for Bernoulli random matrices}]\label{rjl2}Let $\beta\in(0,t_{\pm 1})$ and choose $q_\beta\in(0,\frac{1}{2})$ such that $\alpha_\beta(q_\beta)>0$. Let $N,m,n\in\mathbb{N}$ and $\alpha\in(0,\min\{\alpha_\beta(q_\beta),\frac{1}{12}\})$ 
			be such that
			$$n>\alpha^{-1}\ln[N(N-1)].$$	
Let $\Phi$ be an $n\times m$ Bernoulli random matrix. Then for every $N$-point subset $\{p_1,\dots, p_N\}$ of $\mathbb{R}^m$, we have
				$$\begin{aligned}&\mathbb{P}\left\{\tilde{\theta}\| p_j-p_k\|_2^2\leq\frac{1}{n}\| \Phi_T(p_j-p_k)\|_2^2\leq\tilde{\omega}\| p_j-p_k\|_2^2,\quad\forall |T^c|\leq\beta n, 1\leq j,k\leq N, j\neq k\right\}\\
				\geq &1-N(N-1)\exp(-\alpha n),
				\end{aligned}$$
				$$\begin{aligned}&\mathbb{P}\left\{\frac{\tilde{\theta}}{1-\beta}\| p_j-p_k\|_2^2\leq\frac{1}{|T|}\| \Phi_T(p_j-p_k)\|_2^2\leq\omega\| p_j-p_k\|_2^2,\quad\forall |T^c|\leq\beta n, 1\leq j,k\leq N, j\neq k\right\}\\
				\geq &1-N(N-1)\exp(-\alpha n),
				\end{aligned}$$
				where $\tilde{\theta}$, $\tilde{\omega}$ and $\omega$ are defined as \eqref{ttilde}, \eqref{otilde} and \eqref{o}.
			\end{thm}
			
			\begin{proof}By Lemma\autoref{ltb} and\autoref{eo}, we have
				$$\begin{aligned}&\mathbb{P}\left\{\tilde{\theta}\| p_j-p_k\|_2^2\leq\frac{1}{n}\| \Phi_T(p_j-p_k)\|_2^2\leq\tilde{\omega}\| p_j-p_k\|_2^2,\quad\forall |T^c|\leq\beta n\right\}\\
				\geq &1-2\exp(-\alpha n)
				\end{aligned}$$
				and
				$$\begin{aligned}&\mathbb{P}\left\{\frac{\tilde{\theta}}{1-\beta}\| p_j-p_k\|_2^2\leq\frac{1}{|T|}\| \Phi_T(p_j-p_k)\|_2^2\leq\omega\| p_j-p_k\|_2^2,\quad\forall |T^c|\leq\beta n\right\}\\
				\geq &1-2\exp(-\alpha n)
				\end{aligned}$$
				for every pair $(j,k)$ with $1\leq j,k\leq N$ and $j\neq k$. As there are $\binom{N}{2}=\frac{N(N-1)}{2}$ pairs $(p_j,p_k)$ with $j\neq k$, then by the union bound argument and the assumption $n>\alpha^{-1}\ln[N(N-1)]$, the result follows.\end{proof}

			\section{Summary and Discussion}
			
			In this paper, we have studied the erasure robustness property of Bernoulli random matrices. We've proved the SRIP and a robust version of the Johnson-Lindenstrauss lemma for Bernoulli random matrices. From our analysis, we see that if the portion of rows erasured is smaller than the number $t_{\pm 1}$ which was provided in Theorem\autoref{sripb}, then a Bernoulli random matrix will satisfy the SRIP with high probability. One natural question to ask is that, what happens if the erasure ratio is above $t_{\pm 1}$? Moreover, we are interested in the optimal upper bound of the erasure ratio for the SRIP to hold. We know that an upper bound is $\frac{1}{2}$, but whether or not it's optimal is unknown. It seems that to fully solve this problem, significantly new ideas and techniques are required.

			\section*{Acknowledgements}
			I would like to express my sincere gratitude to my supervisor Professor Bin Han, for his patient help and encouragements.

		\end{flushleft}
	\end{spacing}
\end{document}